%% file: main.tex
\documentclass[11pt]{amsart}
\usepackage[T1]{fontenc}
\usepackage{lmodern}
\usepackage[expansion=false]{microtype}
\usepackage[a4paper,margin=27.5mm]{geometry}
\usepackage{amsmath,amssymb,amsthm,booktabs,array}
\usepackage{needspace,float}
\usepackage[colorlinks=true,linkcolor=blue,citecolor=blue,urlcolor=blue]{hyperref}
\hypersetup{pdftitle={1729 in the Mirror: Reversal and Affine Multiplication},pdfauthor={Chatchawan Panraksa}}
\newtheorem{theorem}{Theorem}[section]
\newtheorem{proposition}[theorem]{Proposition}
\newtheorem{lemma}[theorem]{Lemma}
\newtheorem{corollary}[theorem]{Corollary}
\theoremstyle{remark}
\DeclareMathOperator{\ord}{ord}
\DeclareMathOperator{\lcm}{lcm}
\newcommand{\rev}{R}
\title[1729 in the Mirror]{1729 in the Mirror:\\Reversal and Affine Multiplication}
\author{Chatchawan Panraksa}
\address{Applied Mathematics Program, Mahidol University International College, Salaya, Nakhon Pathom 73170, Thailand}
\email{chatchawan.pan@mahidol.ac.th}
\date{14 September 2026}

\begin{document}
\raggedbottom
\setlength{\emergencystretch}{1em}
\begin{abstract}
The decimal reversals 1729 and 9271 have multiplicative orders of $2$ equal to 36 and 63. Motivated by this symmetry, we classify the affine commutation equation $R_B(hx+1)=hR_B(x)+1$, where $R_B$ reverses base-$B$ digits, for every $B\ge3$, $2\le h<B$ and positive integer $x$ such that neither $x$ nor $hx+1$ ends in zero. Comparing the two carry sequences separates the solutions by their digit lengths. Equal lengths force an odd number of digits alternating between two explicit intervals. When the output gains a digit, coprimality of $h$ and $B+1$ forces a repdigit input with palindromic output. In the noncoprime case, all inputs instead have even length and form a finite family. A slack-variable bijection gives the sharp maximum length, binomial counts at each length, and totals expressed through Fibonacci numbers or powers of two. For reversal pairs of primes already known to share an order index $2\le h<B$, the classification gives an exact digit criterion for reversal of their orders. Finally, an exhaustive computation with independently checkable primality certificates proves that $\{1729,9271\}$ is the unique pair of distinct odd decimal reversals whose orders of $2$ are distinct two-digit reversals of each other.
\end{abstract}
\maketitle

\section{Introduction}

The Hardy--Ramanujan number $1729=1^3+12^3=9^3+10^3$ keeps acquiring curious properties~\cite{Bhakat2017,Luca2026,Taneja2017}. Here is one more. Its decimal reversal is $9271$, and if $o(n)=\ord_n(2)$ denotes the multiplicative order of $2$ modulo an odd integer $n>1$, then the factorizations $1729=7\cdot13\cdot19$ and $9271=73\cdot127$ give
\begin{equation}\label{eq:1729}
 o(1729)=\lcm(3,12,18)=36,
 \qquad o(9271)=\lcm(9,7)=63.
\end{equation}
The mirror image of the number has the mirror image of the order; equivalently, the binary expansion of $1/1729$ repeats every $36$ digits and that of $1/9271$ every $63$. The coincidence is rarer than it looks. Theorem~\ref{thm:classification} shows that, among all pairs of distinct odd integers that are decimal reversals of each other, however large, $\{1729,9271\}$ is the only one whose orders of $2$ are distinct two-digit reversals of each other.

For primes the phenomenon has a structural source. If $p$ is an odd prime and $h=(p-1)/o(p)$ is its order index, then $p=h\,o(p)+1$. Suppose that $p$ and its reversal $r$ are both prime with the same index $h$, and put $x=o(p)$ and $y=o(r)$, so that $r=hy+1$. Then $y$ is the reversal of $x$ precisely when
\[
 R_{10}(hx+1)=h\,R_{10}(x)+1,
\]
where $R_{10}$ reverses decimal digits. For a prime reversal pair already known to share a common index, reversal of the orders is therefore determined by the affine digit equation. The primes $769$ and $967$ have index $2$ and orders $384$ and $483$; the eleven-digit primes $72484572757$ and $75727548427$ have index $3$ and orders $24161524252$ and $25242516142$. In both cases the orders are reversals of each other, and the classification below explains why the digits of such primes must alternate between two prescribed sets.

We answer the digital question completely, in every base. For a positive integer $x$, let $R_B(x)$ reverse its base-$B$ digits and discard leading zeros, and let $L_B(x)$ be its number of digits; thus $R_{10}(1729)=9271$ and $R_{10}(150)=51$. We classify the solutions of
\begin{equation}\label{eq:affine}
 R_B(hx+1)=hR_B(x)+1
\end{equation}
under the standing assumptions
\begin{equation}\label{eq:domain}
 B\ge3,\qquad 2\le h<B,\qquad x\ge1,\qquad
 B\nmid x,\qquad B\nmid hx+1.
\end{equation}
The last two conditions say that neither $x$ nor $hx+1$ ends in the digit $0$, so that reversal is an involution on both; leading digits are nonzero by definition. Throughout, digit positions are counted from the units digit, which has position $0$, while $(a_{L-1},\ldots,a_0)_B$ lists digits from most to least significant.

The whole classification rests on one idea: compare the carries of the two multiplications. Since $x<hx+1<B^{L_B(x)+1}$, the output has either the same number of digits as $x$ or one more. When the lengths agree, the carries of $hx+1$ and of $hR_B(x)+1$ line up position by position and their differences are forced to vanish; the carries must then alternate $1,0,1,\ldots,0$, which makes the length odd and confines the digits to two alternating intervals (Theorem~\ref{thm:division}). When the output gains a digit, the two carry sequences are offset by one place, and the offset is measured by a quantity we call the carry defect. If $\gcd(h,B+1)=1$, the defect can never leave zero and the only solutions are repdigits with palindromic outputs, as in $3\cdot777+1=2332$ (Theorem~\ref{thm:extra}). If $\gcd(h,B+1)>1$, the defect alternates between two levels, the length is even, and adjacent digits must satisfy alternating inequalities (Theorem~\ref{nc:thm:classification}). Those inequalities spend a finite budget of slack, so there are only finitely many such solutions: a slack-variable bijection gives their exact maximum length, the number of solutions of each length as a binomial coefficient, and totals that are one less than a Fibonacci number or one less than a power of two (Theorem~\ref{nc:thm:counts} and Corollary~\ref{nc:cor:totals}). In base ten, $\gcd(h,11)=1$ for every $h\le9$, so every decimal solution is either an odd-length alternating string or a repdigit.

Rivera's discussion of palinpoints~\cite{Palinpoints} credits Joseph L. Pe with the general commutation equation $f(R_{10}(x))=R_{10}(f(x))$. Carry arguments for reversal problems have a long history. Classical reverse multiples satisfy $R_B(N)=hN$, as in $4\cdot2178=8712$; the short-length results of Sutcliffe, Kaczynski, and Klosinski--Smolarski are reviewed by Pudwell~\cite{Pudwell2007}. Young developed paired carry equations and a tree method~\cite{Young1992Multiples,Young1992Trees}, subsequently described by Sloane~\cite[Section 2]{Sloane2014} and extended in Kendrick's graph classifications~\cite{Kendrick2015}. Holt develops carry relations, comparisons of multiplication systems, and finite-state generation for palintiples and permutiples~\cite{Holt2014,Holt2017,Holt2024,Holt2025Finding,Holt2025Multigraph}. His multigraph criterion uses homogeneous multiplication, zero boundary carries and prescribed digit multisets, and does not by itself prescribe reversal order~\cite[Section 4, Theorem 7 and Corollary 2]{Holt2025Multigraph}. For our affine equation, the initial carries equal $1$, and an extra output digit offsets the two carry sequences. We obtain the explicit classification and sharp finite bounds for these conditions using established carry and counting tools. Fibonacci enumeration also occurs for homogeneous reverse multiples~\cite[Section 3.1]{Sloane2014}; the counts here concern the finite noncoprime affine family.

Related papers study different equations, and Radcliffe's fixed-width convention requires a separate comparison. For $0\le N<B^L$, define fixed-width reversal by
\[
 R_{B,L}\!\left(\sum_{i=0}^{L-1}a_iB^i\right)
 =\sum_{i=0}^{L-1}a_iB^{L-1-i},
\]
allowing leading zeros in the width-$L$ representation. Radcliffe~\cite[Theorems 4--5]{Radcliffe2015} classifies $R_{10,L}(N)=2N-1$ with this convention. Faber and Grantham~\cite[pp.~28--29]{FaberGrantham2023} study reversed sum-product pairs in arbitrary bases. Ezerman, Meyer, and Sol\'e~\cite[equation (1), p.~2; Propositions 1--2, p.~4]{Ezerman2015} study reversal distributing over multiplication: carry-free multiplication suffices, and their Proposition 14 supplies a restricted converse for a particular palindrome length. These are methodological precedents for different equations.

Section~\ref{sec:orders} returns to 1729 and 9271, whose orders reverse for a different reason, the least common multiples in \eqref{eq:1729}, and derives the prime criterion; primality and the index remain arithmetic requirements that the criterion assumes rather than proves. Earlier digit identities involving 9271 appear in Bhakat~\cite[Example 2, p.~83]{Bhakat2017} and Taneja~\cite[Sections 19 and 29]{Taneja2017}; Luca's theorem~\cite[Theorem 1]{Luca2026} concerns the different property that $2^s p^3+1$ can be Carmichael for prime $p$ and positive $s$ only when $p=3$.

\section{Equal digit lengths}\label{sec:digits}
Under \eqref{eq:domain}, equal digit lengths align the two carry sequences without a shift, and comparing them forces the carries to alternate.

\begin{theorem}[Equal lengths: odd length and alternating digits]\label{thm:division}
Let $x=\sum_{i=0}^{L-1}a_iB^i$ have $L$ digits.
Then $x$ is a solution of \eqref{eq:affine} with $L_B(hx+1)=L$ if and only if $L$ is odd and
\begin{equation}\label{eq:alternating}
 \begin{cases}
 0\le a_i\le \lfloor(B-2)/h\rfloor,&i\text{ even},\\
 \lceil B/h\rceil\le a_i\le\lfloor(2B-1)/h\rfloor,&i\text{ odd},
 \end{cases}
 \qquad a_0,a_{L-1}\ge1.
\end{equation}
In this case the digits $d_i$ of $hx+1$ are $ha_i+1$ at even positions and $ha_i-B$ at odd positions.
\end{theorem}
\begin{proof}
The carries for multiplication by $h$ followed by addition of $1$ satisfy
\begin{equation}\label{eq:same-carries}
 ha_i+c_i=d_i+Bc_{i+1},\qquad c_0=1,\quad c_L=0,
\end{equation}
where the added $1$ enters as the initial carry $c_0=1$ and $c_L=0$ expresses the equal lengths.
Induction gives $0\le c_i\le h-1$, since $h(B-1)+(h-1)=hB-1$.
If \eqref{eq:affine} holds, then multiplication of $\rev_B(x)$ gives the output digits $d_{L-1-i}$ and carries $e_i$ with $e_0=1$ and $e_L=0$.
The reverse equations, with the same input digit $a_i$, are
\[
 ha_i+e_{L-1-i}=d_i+Be_{L-i}\qquad(0\le i<L).
\]
Subtracting them from \eqref{eq:same-carries} gives
\[
 c_i-e_{L-1-i}=B(c_{i+1}-e_{L-i})\qquad(0\le i<L).
\]
The displayed equality shows that the left side is divisible by $B$. Its absolute value is at most $h-1<B$, so both differences vanish.
In particular, $c_1=e_L=0$; the identities at adjacent indices give $c_{i+2}=c_i$ for $0\le i\le L-2$.
Together with $c_0=1$ and $c_L=0$, these relations force odd $L$ and alternating carries $1,0$.
At even positions \eqref{eq:same-carries} requires $0\le ha_i+1<B$, and at odd positions it requires $B\le ha_i<2B$.
These are precisely \eqref{eq:alternating}.

Conversely, the digit ranges give valid digits in \eqref{eq:same-carries} with alternating carries and zero terminal carry.
Reversal preserves position parity because $L$ is odd, so applying the same operation to $\rev_B(x)$ reverses the output digits.
The two endpoint digits of the output are at least $h+1$ and are therefore nonzero.
This proves \eqref{eq:affine} with the stated digit lengths.
\end{proof}

For the smallest prime example of Section~\ref{sec:orders}, the two multiplications and their carries are
\[
\begin{array}{r@{\quad}ccc}
 \text{carries}&1&0&1\\
 x=384&3&8&4\\ \hline
 2x+1=769&7&6&9
\end{array}
\qquad\qquad
\begin{array}{r@{\quad}ccc}
 \text{carries}&1&0&1\\
 R_{10}(x)=483&4&8&3\\ \hline
 2R_{10}(x)+1=967&9&6&7
\end{array}
\]
where the carry into the units position is the added $1$, and the carries read $1,0,1,0$ from the units end in both multiplications. Put
\[
 A=\left\lfloor\frac{B-2}{h}\right\rfloor,\qquad
 C=\left\lfloor\frac{2B-1}{h}\right\rfloor-\left\lceil\frac Bh\right\rceil+1.
\]
When $h=B-1$, the only case with $A=0$, the endpoint condition excludes equal-length solutions altogether. For $L=1$, the $A$ solutions are all palindromic; at longer odd lengths, the digits can be chosen independently.

\begin{corollary}\label{cor:counts}
For $L=2t+1\ge3$, the number of inputs in Theorem~\ref{thm:division} and the number of their palindromic members are respectively
\begin{equation}\label{eq:counts}
 N_t=A^2(A+1)^{t-1}C^t,\qquad
 P_t=A(A+1)^{\lfloor t/2\rfloor}C^{\lceil t/2\rceil}.
\end{equation}
There are $(N_t-P_t)/2$ unordered pairs of distinct reversals.
The same counts hold for the outputs $hx+1$.
\end{corollary}
\begin{proof}
There are $t+1$ even positions and $t$ odd positions.
The two endpoints have $A$ choices each, the $t-1$ internal even positions have $A+1$ choices each, and the odd positions have $C$ choices each.
Reversal preserves position parity because $L$ is odd. For a palindrome the endpoint pair has $A$ choices, the internal even positions have $\lfloor t/2\rfloor$ free reflection orbits, and the odd positions have $\lceil t/2\rceil$ free reflection orbits.
The nonpalindromic inputs occur in pairs.
The digit map from $x$ to $hx+1$ is injective at each position and is the same at reflected positions, so it preserves palindromicity and reversal pairs.
\end{proof}

When $h\mid B$, the counts simplify through $A=B/h-1$ and $C=B/h$. The output digits are congruent to $1$ modulo $h$ at even positions and to $-B$ at odd positions; when $h\mid B$ the latter residue is $0$. These are counts of digit strings, without primality or order conditions.

\section{When the output gains a digit}\label{sec:extra}\label{sec:noncoprime}

Retain the hypotheses \eqref{eq:domain}, and call a solution of \eqref{eq:affine} an \emph{extra-digit solution} if $L_B(hx+1)=L_B(x)+1$.
The mirror now shifts by one place: the units digit of $hx+1$ becomes the leading digit of $hR_B(x)+1$, so each digit $b_i$ of the output is produced twice, by the forward multiplication at the input digit $a_i$ and by the reversed multiplication at the input digit $a_{i-1}$.
To match the two productions, a state must remember the output digit lying between consecutive input digits together with both carries.
For fixed $B,h$ we therefore form a directed graph whose states are triples $(c,p,b)$ with $0\le c,p<h$ and $0\le b<B$.
There is an edge labeled $u\in\{0,\ldots,B-1\}$ from $(c,p,b)$ to $(c',p',b')$ precisely when
\begin{equation}\label{eq:graph-edge}
 hu+c=b+Bc',\qquad hu+p'=b'+Bp.
\end{equation}
The initial states are $(1,t,t)$ and the terminal states are $(s,1,s)$, for independently chosen $1\le t,s<h$. An \emph{accepting path} is a positive-length path from an initial state to a terminal state.

\begin{lemma}\label{lem:carry-graph}
The positive-length paths from an initial state to a terminal state correspond exactly to solutions of \eqref{eq:affine} with $L_B(hx+1)=L_B(x)+1$.
The edge labels, in path order, are the digits of $x$ from least to most significant.
\end{lemma}
\begin{proof}
Write $x=\sum_{i=0}^{L-1}a_iB^i$ and $hx+1=\sum_{i=0}^{L}b_iB^i$.
Ordinary multiplication gives carries $c_0=1$, $c_L=b_L$ and equations
$ha_i+c_i=b_i+Bc_{i+1}$.
Multiplication of $\rev_B(x)$, with the carries indexed in reverse order as $p_i$, gives
$ha_i+p_{i+1}=b_{i+1}+Bp_i$, with $p_0=b_0$ and $p_L=1$.
All carries lie between 0 and $h-1$ by the same bound as in \eqref{eq:same-carries}.
Thus the states $(c_i,p_i,b_i)$ form a path with the required endpoints.
Conversely, the two edge equations along such a path reconstruct the two multiplication equations and their endpoint carries, and hence \eqref{eq:affine}.

The graph imposes the necessary endpoint digit conditions automatically.
An initial edge with label zero would give $t=1$ and $c'=0$ from the first equation in \eqref{eq:graph-edge}, but the second would then give $b'=p'-B<0$.
A zero label entering a terminal would force $s=1$ and $p=0$ from the second equation, and then $c=b+B$ from the first, contrary to $c<h<B$.
The nonzero parameters $t,s$ give nonzero units and leading digits of the output, which therefore has exactly $L+1$ digits.
The empty path is excluded because $x\ge1$.
\end{proof}

The graph is finite, with at most $h^2B$ states, and its accepting paths of every length account for every extra-digit solution, whether or not $h$ and $B+1$ are coprime. What distinguishes the two regimes is the \emph{carry defect} of a state $(c,p,b)$,
\[
 \eta=b-c-p+1,
\]
which measures how far the two carry sequences have drifted apart. Subtracting the two equations in \eqref{eq:graph-edge} gives, along an edge,
\begin{equation}\label{eq:defect-transition}
 \eta'=\eta+(B-1)(c'-p).
\end{equation}
Both endpoint defects are zero. The following exclusion holds in both regimes.

\Needspace{6\baselineskip}
\begin{lemma}\label{lem:dead-branch}
No accepting path passes through a zero-defect state with $p=0$.
\end{lemma}
\begin{proof}
At such a state $b=c-1$. An outgoing edge labeled $u$ would satisfy
\[
 hu+1=Bc',\qquad b'=hu+p'<B.
\]
Since $hu+1>0$, one has $c'\ge1$; the second inequality gives $hu<B$, so $c'\le1$. Thus $c'=1$, $hu=B-1$, and $p'=0$, giving the state $(1,0,B-1)$. This requires $h\mid B-1$. The new state is not terminal, and a further edge labeled $v$ would require
\[
 hv=B-2+Bc'',\qquad hv+p''<B.
\]
The latter inequality rules out $c''\ge1$, because $B-2+B\ge B$ for $B\ge3$. Hence $c''=0$ and $hv=B-2$, requiring $h\mid B-2$. No $h\ge2$ divides both consecutive integers. Neither state can terminate a path because $p=0$, proving the exclusion.
\end{proof}

When $h$ and $B+1$ are coprime, an accepting path can never leave defect zero, and this forces every digit to be the same.

\begin{theorem}[Extra digit, coprime case: repdigits and palindromes]\label{thm:extra}
Assume that $\gcd(h,B+1)=1$.
There is a unique pair of integers $a,k$ with
\begin{equation}\label{eq:constant-digit}
 1\le a<B,\qquad 1\le k<h,\qquad ha+1=(B+1)k.
\end{equation}
Every solution of \eqref{eq:affine} with $L_B(hx+1)=L_B(x)+1$ has the form
\begin{equation}\label{eq:repdigit}
 x=a\frac{B^L-1}{B-1}.
\end{equation}
If $2k-1<B$, these are solutions for every $L\ge1$; if $2k-1\ge B$, only $L=1$ is a solution.
For each permitted length, the digits are
\[
 x=(a,\ldots,a)_B,\qquad
 hx+1=(k,\underbrace{2k-1,\ldots,2k-1}_{L-1\text{ digits}},k)_B,
\]
listed from most to least significant.
In particular, both integers are palindromes.
\end{theorem}
\begin{proof}
Coprimality gives a unique $a\in\{1,\ldots,B\}$ with $ha\equiv-1\pmod{B+1}$.
The value $a=B$ would imply $h\equiv1\pmod{B+1}$, contrary to $2\le h<B$.
Thus $1\le a<B$, and $k=(ha+1)/(B+1)$ satisfies $1\le k<h$.

Consider a path in Lemma~\ref{lem:carry-graph}.
Every initial state satisfies the relation
\begin{equation}\label{eq:graph-invariant}
 b=c+p-1.
\end{equation}
Suppose that a state satisfies this relation, has $p\ge1$, and has an outgoing edge labeled $u$.
The edge equations become
\begin{equation}\label{eq:graph-reduced}
 hu+1=p+Bc',\qquad b'=B(c'-p)+p+p'-1.
\end{equation}
Since $0\le p+p'-1\le2h-3<2B$ and $0\le b'<B$, the integer $c'-p$ is either $-1$ or 0.
If $c'-p=-1$, then $hu=(B+1)(p-1)$.
Coprimality implies $h\mid p-1$. Since $0\le p-1<h$, this forces $p=1$ and then $u=0$; the second equation gives $b'=p'-B<0$.
Hence $c'=p$, so $hu+1=(B+1)p$.
Uniqueness in \eqref{eq:constant-digit} gives $u=a$ and $p=k$, and the next state is
$(k,p',k+p'-1)$.
It again satisfies \eqref{eq:graph-invariant}.

Lemma~\ref{lem:dead-branch} excludes every zero-defect departure with $p=0$.
Consequently every edge of an accepting path has label $a$, and every departure has $p=k$.

The one-edge path from $(1,k,k)$ to $(k,1,k)$ always exists and gives $x=a$ and $hx+1=(k,k)_B$.
An accepting path of at least two edges must pass through $(k,k,2k-1)$, which is a state exactly when $2k-1<B$.
When that inequality holds, the state has a loop labeled $a$ and an edge labeled $a$ to the terminal, so every positive length occurs.
Reading its output digits gives the displayed formula and proves sufficiency as well as necessity.
When $k=1$, the initial, intermediate, and terminal states coincide; the same loop gives all positive lengths.
\end{proof}

The length restriction occurs already for $B=9$ and $h=7$: then $a=7$, $k=5$, and $2k-1=B$, so $x=7$ gives $(5,5)_9$ and no longer input works. In decimal notation, $\gcd(h,11)=1$ for every $2\le h<10$, so every decimal solution whose output gains a digit is a repdigit with a palindromic output, and Theorem~\ref{thm:division} contains every nonpalindromic solution.

Without coprimality the defect cannot stay at zero either; instead it alternates between two levels, and this is what forces even length. To describe the digits, assume
\begin{equation}\label{nc:eq:parameters}
 d=\gcd(h,B+1)>1,\quad H=h/d,\quad Q=(B+1)/d,\quad M=d-1,
\end{equation}
and define
\begin{equation}\label{nc:eq:thresholds}
 T=\left\lceil\frac BH\right\rceil,\qquad
 S=\left\lfloor\frac{B-2}{H}\right\rfloor,\qquad
 \Delta=T-S,\qquad K_0=2M-T.
\end{equation}
Here $Q\ge2$ and $\gcd(H,Q)=1$. If $K_0\ge0$, then $H\ge2$ and $\Delta\in\{1,2\}$; more precisely, when $H\ge2$, division of $B$ by $H$ shows that $\Delta=2$ if $B\equiv1\pmod H$ and $\Delta=1$ otherwise, and $B>h=dH$ gives $T\ge d+1=M+2$. In words, $d$ measures the obstruction to coprimality, $H$ and $Q$ are the reduced multiplier and the reduced $B+1$, $M$ is the largest possible scaled digit, $T$ and $S$ are the high and low thresholds that adjacent scaled digits will have to meet, $\Delta$ is their gap, and $K_0$ is the slack available to a two-digit solution.

\begin{theorem}[Extra digit, noncoprime case: alternating adjacent sums]\label{nc:thm:classification}
Assume that $d=\gcd(h,B+1)>1$. An integer $x$ is an extra-digit solution if and only if, for some $r\ge1$,
\[
 x=\sum_{i=0}^{2r-1}Qj_iB^i,\qquad 1\le j_i\le M,
\]
where the \emph{scaled digits} $j_i$, defined by $a_i=Qj_i$, satisfy
\begin{equation}\label{nc:eq:alternating}
 j_{i-1}+j_i\ge T\quad(i\text{ odd}),\qquad
 j_{i-1}+j_i\le S\quad(i\text{ even},\ i\ge2).
\end{equation}
In particular, every such input has even length. The output digits $hx+1=\sum_{i=0}^{2r}b_iB^i$ are
\begin{align}\label{nc:eq:output}
 b_0&=Hj_0+1,& b_{2r}&=Hj_{2r-1}+1,\notag\\
 b_i&=H(j_{i-1}+j_i)-B&& (i\text{ odd}),\notag\\
 b_i&=H(j_{i-1}+j_i)+1&& (0<i<2r,\ i\text{ even}).
\end{align}
\end{theorem}

The theorem follows from the behaviour of the defect, which determines both the scaled digits and the parity of the length.
\begin{lemma}[Alternating defects]\label{nc:lem:defects}
On an accepting path with $d>1$, the defects alternate as
\[
 0,-(B-1),0,-(B-1),\ldots,0.
\]
A departure from a zero-defect state has
\[
 a_i=Qj_i,\quad p_i=Hj_i+1,\quad c_{i+1}=Hj_i,
\]
whereas a departure from a negative-defect state has
\[
 a_i=Qj_i,\quad p_i=Hj_i,\quad c_{i+1}=Hj_i+1.
\]
In both cases $1\le j_i\le d-1$.
\end{lemma}
\begin{proof}
At a zero-defect state with $p\ge1$, the edge equations read
\[
 hu+1=p+Bc',\qquad b'=B(c'-p)+p+p'-1.
\]
Because $0\le p+p'-1\le2h-3<2B$ and $0\le b'<B$, we have $c'-p\in\{-1,0\}$. The value zero would give $hu+1=(B+1)p$, impossible modulo $d>1$. Therefore $c'=p-1$ and
\[
 hu=(B+1)(p-1),\qquad \eta'=-(B-1).
\]
Since $\gcd(H,Q)=1$, there is an integer $j$ with $u=Qj$, $p-1=Hj$. The value $j=0$ would give $p=1$ and $b'=p'-B<0$. Thus $j\ge1$, while $u<B$ gives $j\le d-1$.

At a state of defect $-(B-1)$, one has $b=c+p-B$. Thus $p\ge B-h+1\ge2$. Put $q=c'-p$. The edge equations become
\[
 hu=(B+1)p+B(q-1),\qquad b'=B(q-1)+p+p'.
\]
The digit and carry bounds force $q\in\{0,1\}$. Reduction of the first equation modulo $d$ gives $q\equiv1\pmod d$, so $q=1$. Consequently $hu=(B+1)p$, $c'=p+1$, and $\eta'=0$. Again $u=Qj$, $p=Hj$, and $1\le j\le d-1$.

Lemma~\ref{lem:dead-branch} excludes zero-defect states with $p=0$. Induction from the initial state now proves the lemma; terminal defect zero forces even path length.
\end{proof}

\begin{proof}[Proof of Theorem~\ref{nc:thm:classification}]
Lemma~\ref{nc:lem:defects} gives $L=2r$ and $a_i=Qj_i$ with $1\le j_i\le M$. Reading $b_i=c_i+p_i-1$ at even positions and $b_i=c_i+p_i-B$ at odd positions yields \eqref{nc:eq:output}. For odd $i$, the lower digit bound is precisely $j_{i-1}+j_i\ge T$; its upper bound is automatic because $H(j_{i-1}+j_i)-B\le2h-2H-B<B$. For interior even $i$, the upper digit bound is precisely $j_{i-1}+j_i\le S$, and its lower bound is automatic. This proves necessity.

Conversely, suppose \eqref{nc:eq:alternating} holds. The first adjacent sum implies $T\le2M$, hence $K_0\ge0$. This forces $H\ge2$: if $H=1$, then $B+1=dQ\ge2d$, so $T=B\ge2d-1>2M$. Thus $Hj_i+1\le h-H+1\le h-1<B$. Every $a_i=Qj_i$ is a positive digit since $Q(d-1)=B+1-Q\le B-1$. All the $b_i$ in \eqref{nc:eq:output} are digits, and $b_0,b_{2r}>0$.

The forward multiplication follows directly from \eqref{nc:eq:output} with $c_0=1$ and
\[
 c_{i+1}=\begin{cases}Hj_i,&i\text{ even},\\Hj_i+1,&i\text{ odd}.
 \end{cases}
\]
Reversing an even-length sequence $j_0,\ldots,j_{2r-1}$ preserves the alternating adjacent-sum conditions. Applying the same output formulas to this reversed sequence gives precisely $b_{2r},\ldots,b_0$. Therefore $hR_B(x)+1=R_B(hx+1)$, completing sufficiency.
\end{proof}

\section{Finiteness and exact counts}\label{sec:enumeration}

Retain $d>1$ and the parameters \eqref{nc:eq:parameters}--\eqref{nc:eq:thresholds}. The alternating bounds \eqref{nc:eq:alternating} can be read as a budget. Scaled digits cannot exceed $M$, so once the first high sum has been paid for, the two endpoint digits leave at most $2M-T=K_0$ units of slack, and every further pair of digits costs at least $\Delta=T-S$ more, because a high sum of at least $T$ must be followed by a low sum of at most $S$. When the budget runs out the solutions stop, which is why there are finitely many; and counting the ways of spending it, by a bijection with compositions, gives their exact number at every length.

\begin{theorem}[Sharp lengths and exact counts]\label{nc:thm:counts}
For $r\ge1$, put $K_r=K_0-(r-1)\Delta$. There are no extra-digit solutions of length $2r$ if $K_r<0$. If $K_r\ge0$, the numbers of inputs and of palindromic inputs of length $2r$ are, respectively,
\begin{equation}\label{nc:eq:counts}
 E_r=\binom{K_r+2r}{2r},\qquad
 P_r=\binom{\lfloor K_r/2\rfloor+r}{r}.
\end{equation}
The number of unordered nonpalindromic reversal pairs is $(E_r-P_r)/2$.
There are no extra-digit solutions at all when $K_0<0$. Otherwise their exact maximum input length is
\begin{equation}\label{nc:eq:maxlength}
 L_{\max}=2\left(1+\left\lfloor\frac{K_0}{\Delta}\right\rfloor\right),
\end{equation}
and every positive even length not exceeding this bound occurs.
\end{theorem}

\begin{proof}[Proof of Theorem~\ref{nc:thm:counts}]
Set $e_i=j_{2i}$ and $o_i=j_{2i+1}$ for $0\le i<r$. Define nonnegative slacks
\begin{align*}
 U&=M-e_0,& V&=M-o_{r-1},\\
 \alpha_i&=e_i+o_i-T&&(0\le i<r),\\
 \beta_i&=S-o_i-e_{i+1}&&(0\le i<r-1).
\end{align*}
Summation gives the single identity
\begin{equation}\label{nc:eq:slacks}
 U+V+\sum_{i=0}^{r-1}\alpha_i+\sum_{i=0}^{r-2}\beta_i
 =2M-T-(r-1)(T-S)=K_r.
\end{equation}
Conversely, any nonnegative integral solution of \eqref{nc:eq:slacks} reconstructs exactly one digit sequence by
\[
 e_0=M-U,\qquad o_i=T-e_i+\alpha_i,\qquad
 e_{i+1}=S-o_i-\beta_i.
\]
Indeed, the recursion gives $e_{i+1}=e_i-\Delta-\alpha_i-\beta_i$, so the even subsequence decreases. Similarly, $o_{i+1}-o_i=\Delta+\beta_i+\alpha_{i+1}>0$, so the odd subsequence increases. Equation~\eqref{nc:eq:slacks} makes its final odd entry equal to $M-V$. Consequently every $e_i\le M$ and every $o_i\le M$. The high-sum equalities imply $e_i,o_i\ge T-M\ge2$, so all the required lower bounds also hold. Thus there are no unrecorded digit restrictions in the slack parametrization.

There are $2r+1$ slack variables. Stars and bars now gives $E_r=\binom{K_r+2r}{2r}$ when $K_r\ge0$, and no solutions otherwise. Reversal of the digit sequence reverses the ordered slack vector
\[
 (U,\alpha_0,\beta_0,\alpha_1,\beta_1,\ldots,
      \beta_{r-2},\alpha_{r-1},V).
\]
For $r=1$, the vector is simply $(U,\alpha_0,V)$. A palindromic vector has $r$ pairs of equal entries and one unpaired middle entry. Thus its count is
\[
 \sum_{j=0}^{\lfloor K_r/2\rfloor}\binom{j+r-1}{r-1}
 =\binom{\lfloor K_r/2\rfloor+r}{r}.
\]
This proves \eqref{nc:eq:counts} and the orbit count. Solving $K_r\ge0$ gives \eqref{nc:eq:maxlength}. Existence at every allowed length follows either from the slack bijection or from the explicit choice
\[
 j_{2i}=M-i\Delta,\qquad j_{2i+1}=T-M+i\Delta
 \qquad(0\le i<r).
\]
The length bound ensures these lie in $[1,M]$; high adjacent sums equal $T$ and low sums equal $S$.
\end{proof}

\begin{corollary}[Parameter criteria]\label{nc:cor:existence}
Extra-digit solutions exist if and only if
\[
 B\le2h-\frac{2h}{d}.
\]
Nonpalindromic extra-digit solutions exist if and only if
\[
 B\le2h-\frac{3h}{d}.
\]
When $K_0=0$, there is exactly one extra-digit solution: the two-digit repdigit $x=(QM,QM)_B$.
\end{corollary}
\begin{proof}
Existence is equivalent to $K_0\ge0$, or $\lceil B/H\rceil\le2M$, which is the first inequality. If $K_0=0$, Theorem~\ref{nc:thm:counts} gives $E_1=P_1=1$ and no larger lengths. If $K_0\ge1$, there is a nonpalindromic two-digit solution: the inequalities $T\le2M-1$ permit $(j_0,j_1)=(M,M-1)$. Conversely, $K_0\le0$ permits none. Thus nonpalindromic existence is equivalent to $T\le2M-1$, giving the second inequality.
\end{proof}

\begin{corollary}[Total counts]\label{nc:cor:totals}
Suppose $K=K_0\ge0$, and normalize the Fibonacci sequence by $F_0=0$, $F_1=1$. The total numbers $E$ and $P$ of extra-digit inputs and palindromic extra-digit inputs, summed over all lengths, are
\[
\begin{array}{c|cc}
 & E & P\\\hline
 \Delta=1 & F_{2K+3}-1 & F_{K+3}-1\\
 \Delta=2 & 2^{K+1}-1 & 2^{\lfloor K/2\rfloor+1}-1.
\end{array}
\]
The total number of unordered nonpalindromic pairs is $(E-P)/2$.
\end{corollary}

\begin{proof}[Proof of Corollary~\ref{nc:cor:totals}]
For $\Delta=2$, \eqref{nc:eq:counts} becomes $E_r=\binom{K+2}{2r}$, whose positive-even-index sum is $2^{K+1}-1$. For $\Delta=1$, it becomes $E_r=\binom{K+r+1}{2r}$. The elementary Fibonacci identity $F_{n+1}=\sum_{j\ge0}\binom{n-j}{j}$ with $n=2K+2$ and the substitution $j=K-r+1$ gives $\sum_{r\ge1}E_r=F_{2K+3}-1$.

For the palindromic total, treating $K$ as a nonnegative integer parameter gives the generating function
\begin{align*}
 \sum_{K\ge0}P(K)z^K
 &=\sum_{r\ge1}z^{(r-1)\Delta}
   \sum_{k\ge0}\binom{\lfloor k/2\rfloor+r}{r}z^k\\
 &=\frac{1}{(1-z)(1-z^2-z^\Delta)}.
\end{align*}
When $\Delta=1$, its coefficient is $\sum_{j=0}^K F_{j+1}=F_{K+3}-1$. When $\Delta=2$, its coefficient is $\sum_{j=0}^{\lfloor K/2\rfloor}2^j=2^{\lfloor K/2\rfloor+1}-1$.
\end{proof}

For $(B,h)=(14,10)$, the parameters are
\[
 d=5,\ H=2,\ Q=3,\ M=4,\ T=7,\ S=6,\ \Delta=1,\ K_0=1.
\]
There are exactly three inputs of length two and one of length four:
\[
 (9,12)_{14},\qquad(12,9)_{14},\qquad(12,12)_{14},\qquad
 (12,9,9,12)_{14}.
\]
The first two form the only nonpalindromic reversal pair, and the other two are palindromes; no further extra-digit solutions occur. The first input is $138$ in decimal, with output $(7,0,9)_{14}=1381$; the four-digit input has output $(9,0,13,0,9)_{14}$. It is the case $m=5$ of the following family, which realizes the nonpalindromic criterion for every $m\ge5$.

\begin{proposition}\label{prop:noncoprime}
For every integer $m\ge5$, put $B=3m-1$, $h=2m$, and $x=(3m-6,3m-3)_B$.
Then $x$ is a nonpalindromic solution of \eqref{eq:affine}, and
\begin{equation}\label{eq:counterfamily}
 hx+1=(2m-3,m-5,2m-1)_B
\end{equation}
has one more digit than $x$.
\end{proposition}
\begin{proof}
Here $d=m$, $H=2$, $Q=3$, and $M=m-1$. The scaled digits, starting at the units position, are $j_0=m-1$ and $j_1=m-2$. They lie in $[1,M]$, and
\[
 H(j_0+j_1)=4m-6\ge3m-1=B
\]
precisely when $m\ge5$. Theorem~\ref{nc:thm:classification} therefore applies. Its output formulas give $b_0=2m-1$, $b_1=m-5$, and $b_2=2m-3$, which is \eqref{eq:counterfamily}. The input digits differ by 3 and the output endpoint digits differ by 2, so neither integer is palindromic.
\end{proof}

\section{From 1729 to reversed primes}\label{sec:orders}\label{sec:classification}\label{sec:primes}

We return to the mirror in which the paper began. Abbreviate $R=R_{10}$ and require both reversal identities
\begin{equation}\label{eq:mutual}
 R(o(n))=o(R(n)),\qquad R(o(R(n)))=o(n);
\end{equation}
the second excludes orders ending in zero, whose reversal would lose a digit. The moduli may be arbitrarily large, yet restricting the orders to two digits makes the problem finite, because a modulus with $o(n)=k$ divides $2^k-1$, and there are only finitely many such divisors.

\begin{theorem}[Uniqueness for two-digit orders]\label{thm:classification}
Suppose that $n$ and $\rev(n)$ are distinct odd integers greater than 1.
If $o(n)$ and $o(\rev(n))$ are distinct two-digit decimal integers
satisfying \eqref{eq:mutual}, then
\[
 \{n,\rev(n)\}=\{1729,9271\}.
\]
\end{theorem}

Everything turns on the distinction between $n$ dividing $2^k-1$ and $2$ having exact order $k$ modulo $n$. We use the following elementary criterion, which will also certify the prime examples.

\begin{lemma}\label{lem:order}
For an odd integer $n>1$ and a positive integer $k$, the equality
$o(n)=k$ holds if and only if
\begin{equation}\label{eq:order-test}
 n\mid 2^k-1,\qquad
 2^{k/q}\not\equiv1\pmod n\quad\text{for every prime }q\mid k.
\end{equation}
\end{lemma}
\begin{proof}
The divisibility condition says that $o(n)$ divides $k$. If this divisor
is proper, then $k/o(n)>1$ has a prime divisor $q$, and $o(n)$ divides
$k/q$. This contradicts the corresponding inequality in
\eqref{eq:order-test}. Conversely, order $k$ gives the divisibility and
rules out every smaller exponent $k/q$.
\end{proof}

For a fixed order $k$, this criterion restricts the modulus to a divisor of $2^k-1$. Enumerating those divisors, reversing the surviving moduli and testing the reversed order yields a finite classification.

\begin{proof}[Proof of Theorem~\ref{thm:classification}]
Interchange the two moduli, if necessary, so that $k=o(n)$ is smaller
than $\ell=o(\rev(n))$. Both orders have two digits and mutually reverse,
so there are digits $1\le a<b\le9$ such that
\begin{equation}\label{eq:exponents}
 k=10a+b,\qquad \ell=10b+a.
\end{equation}
Thus there are exactly $\binom92=36$ exponent pairs. Every possible
smaller-order modulus is a divisor of one of the 36 integers $2^k-1$.

For each $k$ in \eqref{eq:exponents}, factor $2^k-1$ completely, generate
every positive divisor, and apply \eqref{eq:order-test}. For each
surviving divisor $n$, put $r=\rev(n)$ and require $r>1$, $r$ odd,
$r\ne n$, and $2^\ell\equiv1\pmod r$. Table~\ref{tab:exhaustive}
records the divisor and exact-order counts. Across all rows there are
2404 divisor incidences, counted separately for each exponent row and including the divisor 1 in each row, and
2013 incidences with the required exact smaller order. The reversal
divisibility test leaves only
\begin{equation}\label{eq:weak}
 (n,r,k,\ell)=(37,73,36,63),\quad (1729,9271,36,63).
\end{equation}
The complete factorizations, certificates, and executable checks
underlying this finite step are described in Appendix~\ref{sec:reproduction}.

The first tuple in \eqref{eq:weak} does not have exact reversed order:
$o(73)=9$. For the second, the residues
\[
\begin{array}{c|rrrr}
n&k&2^k\bmod n&\multicolumn{2}{c}{\text{shortened-exponent residues}}\\\hline
1729&36&1&2^{18}\equiv1065&2^{12}\equiv638\\
9271&63&1&2^{21}\equiv1906&2^9\equiv512
\end{array}
\]
prove the exact orders 36 and 63 by Lemma~\ref{lem:order}. This also
verifies existence. Exhaustiveness follows from
\eqref{eq:exponents} and the enumeration of every divisor, with no
separate bound on either modulus.
\end{proof}

The composite pair above is governed by its prime-factor orders. For primes, the affine classification gives a structural explanation: a common index $h=(p-1)/o(p)$ turns the two orders into inputs for the same affine map. The following criterion assumes that this arithmetic index actually occurs.

\begin{corollary}\label{cor:indices}
Let $B\ge3$, and let $p$ and $r=\rev_B(p)$ be distinct odd primes.
If their base-2 orders are mutual base-$B$ reversals, then the two orders cannot both have index 1.
Suppose that both orders have a common index $h$ with $2\le h<B$, and put $x=(p-1)/h$ and $d=\gcd(h,B+1)$. Then the orders are mutual reversals if and only if $x$ either has odd length and satisfies \eqref{eq:alternating} with positive endpoint digits, or has even length and satisfies the digit conditions of Theorem~\ref{nc:thm:classification} with $d>1$.
\end{corollary}
\begin{proof}
Since reversal gives another prime greater than 1, the last digit of $p$ is nonzero.
Subtracting 1 changes only that digit, and
\[
 \rev_B(p-1)=r-B^{L_B(p)-1}.
\]
If both orders had index 1, equality to $r-1$ would force $L_B(p)=1$, contrary to distinctness.

For common index $h$, put $x=(p-1)/h$ and $y=(r-1)/h$.
If the orders are mutual reversals, neither ends in zero and $y=\rev_B(x)$, so
$\rev_B(hx+1)=r=h\rev_B(x)+1$.
When $d=1$, Theorem~\ref{thm:extra} makes an extra-digit output palindromic, contrary to $p\ne r$. Theorem~\ref{thm:division} therefore gives the first alternative. When $d>1$, Theorems~\ref{thm:division} and~\ref{nc:thm:classification} give the two alternatives, and Theorem~\ref{nc:thm:counts} bounds the extra-digit length. Conversely, either digit alternative gives \eqref{eq:affine}, so $y=\rev_B(x)$.
The nonzero endpoints give the reverse identity as well.
\end{proof}

Theorem~\ref{nc:thm:counts} adds a length restriction to the second alternative: it requires $K_0\ge0$ and gives $L_B(p)\le L_{\max}+1$, with $L_{\max}$ as in \eqref{nc:eq:maxlength}. Hence if $d=1$, or if $d>1$ and $K_0<0$, only the first alternative can occur, and for fixed noncoprime parameters with $K_0\ge0$ every such prime pair of length greater than $L_{\max}+1$ satisfies the first alternative.

For decimal index 2, the output digits alternate odd and even, with odd length and endpoints at least 3.
Thus 769 and 967 have the required pattern and have exact orders 384 and 483 by \eqref{eq:order-test}.
Two further consequences of Corollary~\ref{cor:indices} in base ten deserve to be stated. Since $\gcd(h,11)=1$ for every $h\le9$, two distinct mirror-image decimal primes with a common index $h\le9$ can have mirror-image orders only if they have an odd number of digits; and since $A=\lfloor8/9\rfloor=0$ for $h=9$, distinct mirror-image decimal primes sharing the index $9$ never have mirror-image orders.
Theorem~\ref{thm:division} applies equally to indices that do not divide $10$.
For index 3, the even-position digits of $p$ belong to $\{1,4,7\}$, the odd-position digits belong to $\{2,5,8\}$, and the endpoints belong to $\{4,7\}$.
If both reversed integers are multidigit primes, their endpoints must therefore be 7.
An explicit example, certified in Appendix~\ref{sec:reproduction}, is
\begin{equation}\label{eq:index-three}
 \begin{aligned}
 p&=72484572757,& o(p)&=24161524252=(p-1)/3,\\
 r&=75727548427,& o(r)&=25242516142=(r-1)/3.
 \end{aligned}
\end{equation}
Both the primes and the two displayed orders are mutual decimal reversals.

An exhaustive search over both moduli below $10^5$ gives
Table~\ref{tab:small}. The search includes every pair
$3\le n<\rev(n)<100000$ with both moduli odd, tests their exact
orders, and imposes both identities in \eqref{eq:mutual} with
distinct orders. All seven prime pairs have index $2$, and each shows
the odd length and alternating digit parities that
Corollary~\ref{cor:indices} prescribes, beginning and ending with an
odd digit at least $3$. The composite pair 1729, 9271 has a different
explanation through its small prime-factor orders.

\begin{table}[!htbp]\centering\small
\caption{All pairs $n<R(n)<10^5$ of odd integers with distinct, mutually reversed orders of $2$.}
\label{tab:small}
\begin{tabular}{rrrrl}
\toprule
$n$&$\rev(n)$&$o(n)$&$o(\rev(n))$&Type\\\midrule
769&967&384&483&both prime\\
1729&9271&36&63&both composite\\
30367&76303&15183&38151&both prime\\
30983&38903&15491&19451&both prime\\
34367&76343&17183&38171&both prime\\
34583&38543&17291&19271&both prime\\
38327&72383&19163&36191&both prime\\
92369&96329&46184&48164&both prime\\
\bottomrule\end{tabular}
\end{table}

For prime reversal pairs already known to share a common index $2\le h<B$, the classification gives an exact digit criterion for reversal of their orders. Whether such primes exist, and with which index, remains an arithmetic question. For each index $2\le h\le8$ the criterion allows infinitely many decimal digit strings, and Table~\ref{tab:small} and \eqref{eq:index-three} give prime examples for indices $2$ and $3$, but we cannot prove that any index yields infinitely many prime pairs. The sieve results of Dartyge, Martin, Rivat, Shparlinski, and Swaenepoel~\cite{Dartyge2024} concern binary reversal without an order-index condition and do not establish this assertion. Theorem~\ref{thm:classification} excludes every composite pair other than $\{1729,9271\}$ with distinct two-digit reversed orders, and Table~\ref{tab:small} contains no further composite pair below $10^5$. We do not settle whether further composite pairs with longer orders exist.

\appendix
\section{Certificates and computational details}\label{sec:reproduction}

The exhaustive calculation in Theorem~\ref{thm:classification} uses complete factorizations of the 36 integers $2^k-1$ in \eqref{eq:exponents}. Their prime factors and recursive dependencies are certified by the following criterion.

\begin{lemma}[Lucas primality criterion]\label{lem:lucas}
Let $p>2$ be an integer, with a complete prime factorization of $p-1$. If an integer $a$ satisfies
\begin{equation}\label{eq:lucas}
 a^{p-1}\equiv1\pmod p,\qquad
 \gcd(a^{(p-1)/q}-1,p)=1
 \quad\text{for every prime }q\mid p-1,
\end{equation}
then $p$ is prime.
\end{lemma}
\begin{proof}
For any prime divisor $t$ of $p$, the first congruence makes $a$ invertible modulo $t$, with order dividing $p-1$. For each prime $q\mid p-1$, the gcd condition prevents that order from dividing $(p-1)/q$. Hence its $q$-adic exponent equals the full exponent of $q$ in $p-1$. Its order is therefore $p-1$, which divides $t-1$. Thus $t\ge p$, and $p$ is prime.
\end{proof}

The computational supplement contains \path{certificates.json}, which supplies the complete factorizations and witnesses in Lemma~\ref{lem:lucas}, with recursion terminating at 2. Its standard-library verifier \path{check_classification.py} checks every product, witness and divisor exponent combination using exact integers. The 105 prime certificates support the exhaustive counts in Table~\ref{tab:exhaustive}: 2404 divisor incidences and 2013 exact-order incidences over 36 exponent pairs, followed by the two weak tuples in \eqref{eq:weak} and the unique final solution. The optional certificate generator is unnecessary for verification; certifying the final pair alone would not prove exhaustion.

For the index-3 example in \eqref{eq:index-three}, the certificates are elementary.
Complete trial division proves the primality of $p,r$ and of the factors in
\[
 24161524252=2^2\cdot6040381063,\qquad
 25242516142=2\cdot37\cdot991\cdot344213.
\]
For each displayed order $q$, the residue $2^q$ modulo the corresponding prime is 1.
For $p$, the shortened-exponent residues for prime factors 2 and 6040381063 are $p-1$ and 16.
For $r$, those for 2, 37, 991, and 344213 are respectively
$r-1$, 44316499215, 858223607, and 47864508703.
Lemma~\ref{lem:order} proves the exact orders.

The remaining scripts reproduce Table~\ref{tab:small}, check the index-3 certificates and test the digital classifications over their recorded finite ranges. The noncoprime control compares direct reversal with independent digit-transition counts rather than the slack bijection. These finite controls supplement the unrestricted proofs in Sections~\ref{sec:digits}--\ref{sec:enumeration}. The repository README gives exact commands, tested environments, expected outputs and file hashes. The four principal verifiers use the Python standard library; the bounded search and optional certificate generator also use SymPy~\cite{SymPy2017}. All reported results were reproduced with Python 3.12.14 and SymPy 1.14.0, using deterministic inputs and exact integers.

\begin{table}[H]
\centering\small
\caption{The complete smaller-order domain. Here $D_k$ is the number of
positive divisors of $2^k-1$, and $E_k$ is the number having exact
base-2 order $k$. The only row with survivors after the reversal
divisibility test is $k=36$, giving \eqref{eq:weak}.}
\label{tab:exhaustive}
\input{classification_table.tex}
\end{table}

The scripts, primality certificates and recorded outputs are available under the MIT license at
\url{https://github.com/chatchawanpan-dev/1729-reversal-computations}. The computations in this paper use \href{https://github.com/chatchawanpan-dev/1729-reversal-computations/tree/bc7db2d920fdf9fde3484f4bceb3183eef999f6b}{version 1.0.0 (commit \texttt{bc7db2d920fd})}.

\bibliographystyle{ramanujan-references}
\bibliography{references}
\end{document}

%% file: classification_table.tex
\begin{tabular}{rrrr@{\qquad}rrrr}
\toprule
$k$&$\ell$&$D_k$&$E_k$&$k$&$\ell$&$D_k$&$E_k$\\
\midrule
12 & 21 & 24 & 16 & 37 & 73 & 4 & 3 \\
13 & 31 & 2 & 1 & 38 & 83 & 8 & 5 \\
14 & 41 & 8 & 5 & 39 & 93 & 16 & 13 \\
15 & 51 & 8 & 5 & 45 & 54 & 64 & 54 \\
16 & 61 & 16 & 8 & 46 & 64 & 16 & 11 \\
17 & 71 & 2 & 1 & 47 & 74 & 8 & 7 \\
18 & 81 & 32 & 24 & 48 & 84 & 768 & 664 \\
19 & 91 & 2 & 1 & 49 & 94 & 4 & 2 \\
23 & 32 & 4 & 3 & 56 & 65 & 256 & 188 \\
24 & 42 & 96 & 68 & 57 & 75 & 16 & 13 \\
25 & 52 & 8 & 6 & 58 & 85 & 64 & 55 \\
26 & 62 & 8 & 5 & 59 & 95 & 4 & 3 \\
27 & 72 & 8 & 4 & 67 & 76 & 4 & 3 \\
28 & 82 & 64 & 54 & 68 & 86 & 128 & 118 \\
29 & 92 & 8 & 7 & 69 & 96 & 16 & 11 \\
34 & 43 & 8 & 5 & 78 & 87 & 192 & 167 \\
35 & 53 & 16 & 13 & 79 & 97 & 8 & 7 \\
36 & 63 & 512 & 462 & 89 & 98 & 2 & 1 \\
\bottomrule
\end{tabular}